\documentclass[reqno]{amsart}
\usepackage{amsmath}
\usepackage{amssymb}
\usepackage{amsthm}
\usepackage{cases}
\usepackage{mathtools}
\mathtoolsset{showonlyrefs=true}
\usepackage{mathrsfs}
\usepackage{hyperref}

\makeatletter
\@namedef{subjclassname@2020}{%
  \textup{2020} Mathematics Subject Classification}
\makeatother

\title{A proof of the second Rogers-Ramanujan identity via Kleshchev multipartitions}

\author{Shunsuke Tsuchioka}
\address{Department of Mathematical and Computing Sciences, Tokyo Institute of Technology, Tokyo 152-8551, Japan}
\email{tshun@kurims.kyoto-u.ac.jp}

\date{May 16, 2022}

\usepackage{tikz}

\usetikzlibrary{arrows}
\tikzstyle{every picture}+=[remember picture]
\tikzstyle{na} = [baseline=-.5ex]
\tikzstyle{mine}= [arrows={angle 90}-{angle 90},thick]

\def\Llleftarrow{%
\lower2pt\hbox{\begingroup
\tikz
\draw[shorten >=0pt,shorten <=0pt] (0,3pt) -- ++(-1em,0) (0,1pt) -- ++(-1em-1pt,0) (0,-1pt) -- ++(-1em-1pt,0) (0,-3pt) -- ++(-1em,0) (-1em+1pt,5pt) to[out=-105,in=45] (-1em-2pt,0) to[out=-45,in=105] (-1em+1pt,-5pt);
\endgroup}
}

\newtheorem{Thm}{Theorem}[section]

\newtheorem{Prop}[Thm]{Proposition}

\newtheorem{Rem}[Thm]{Remark}
\newtheorem{Cor}[Thm]{Corollary}

\DeclareMathOperator{\CH}{ch}

\DeclareMathOperator{\STR}{\mathsf{Str}}
\DeclareMathOperator{\PAR}{\mathsf{Par}}

\usepackage{graphicx}

\begin{document}
\maketitle

\begin{abstract}
We give another proof of the second Rogers-Ramanujan identity by Kashiwara crystals.
\end{abstract}

%\tableofcontents
\section{Introduction}
In ~\cite{LM}, Lepowsky and Milne observed a similarity
between
the characters 
of the level 3 standard modules of the affine Lie algebra of type $A^{(1)}_{1}$
\begin{align}
\CH V(2\Lambda_0+\Lambda_1) &= \frac{1}{(q;q^2)_{\infty}}\frac{1}{(q,q^4;q^5)_{\infty}},\\
\CH V(3\Lambda_0) &= \frac{1}{(q;q^2)_{\infty}}\frac{1}{(q^2,q^3;q^5)_{\infty}}.
\label{chsimi}
\end{align}
and the infinite products of the Rogers-Ramanujan identities 
\begin{align}
\sum_{n\geq 0}\frac{q^{n^2}}{(q;q)_n} &= \frac{1}{(q,q^4;q^5)_{\infty}},\\
\sum_{n\geq 0}\frac{q^{n(n+1)}}{(q;q)_n} &= \frac{1}{(q^2,q^3;q^5)_{\infty}}.
\label{RRiden}
\end{align}
Here, the $q$-Pochhammer symbols are defined for $n\in\mathbb{Z}_{\geq0} \sqcup \{\infty\}$ as follows.
\begin{align*}
(a;q)_n = \prod_{0\leq j<n} (1-aq^j),\quad
(a_1,\dots,a_k;q)_{n}= (a_1;q)_n \cdots (a_k;q)_n.
\end{align*}

Later, Lepowsky and Wilson promoted the observation to a vertex operator proof
and gave a Lie theoretic interpretation of the infinite sums in the Rogers-Ramanujan identities~\cite{LW3}.
The goal of this paper is to show that a result of Kashiwara crystals 
which is motivated by the representation theory of Hecke algebras~\cite[Corollary 9.6]{AKT}
promotes the equality \eqref{chsimi} into a proof of the second Rogers-Ramanujan identity \eqref{RRiden}.
Note that it is well-known that the Rogers-Ramanujan identities and the solvable lattice models 
from which the quantum groups originated are related (see ~\cite[Chapter 8]{An2}).
Several relationships between Rogers-Ramanujan type identities and Kashiwara crystals 
are also known (see ~\cite{DL} and the references therein).
The author was inspired by a recent work of Corteel which
gave a proof of \eqref{RRiden} %for the second Rogers-Ramanujan identity using
using the cylindric partitions and the Robinson-Schensted-Knuth correspondence~\cite{Cor}.

%In ~\cite{Har}, Hardy mentioned seven proofs for the Rogers-Ramanujan identities, which now are regarded as classical.
%Recently, Corteel gave a proof for the second Rogers-Ramanujan identity using
%cylindric partitions and Robinson-Schensted-Knuth correspondence~\cite{Cor}.
%It would be interesting if one can make a connection between these proofs.
%It also would be interesting if the method 
%works for other (conjectural) partition identities
%such as Andrews-Gordon, Capparelli, Nandi, Kanade-Russell, and so on.

%\newpage

\section{The main result}
A partition (resp. strict partition) is a weakly (resp. strictly) decreasing 
sequence $\lambda=(\lambda_1,\dots,\lambda_{\ell})$
of positive integers, i.e., $\lambda_1\geq\dots\geq\lambda_{\ell}\geq1$ (resp. $\lambda_1>\dots>\lambda_{\ell}\geq1$). 
We denote the set of partitions (resp. strict partitions) by $\PAR$ (resp. $\STR$).
We also denote the size $\lambda_1+\dots+\lambda_{\ell}$ (resp. the length $\ell$) of $\lambda$ by $|\lambda|$ (resp. $\ell(\lambda)$).
When $\lambda$ is empty (i.e., $\ell(\lambda)=0$), we put $\lambda_1=0$.

%For $p\geq 2$, 
%Thanks to Misra-Miwa realization~\cite{MM},
%the set of strict partitions $\STR$ have a $A^{(1)}_{1}$-crystal structure 
%that is isomorphic to the level 1 crystal $B(\Lambda_0)$.

\begin{Thm}[{\cite[(The transposed version of) Proposition 9.7]{AKT}}]
Let $k\geq 1$. Under the $A^{(1)}_{1}$-crystal isomorphism $\STR\cong B(\Lambda_0)$ due to Misra-Miwa~\cite{MM},
the canonical image $B(k\Lambda_0)$ in the tensor product $B(\Lambda_0)^{\otimes k}$
coincides with 
\begin{align*}
S_k = \{
\boldsymbol{\lambda}=(\lambda^{(1)},\dots,\lambda^{(k)})\in\STR^k\mid 
\ell(\lambda^{(i)})\geq (\lambda^{(i+1)})_1\textrm{ for $1\leq i<k$}
\}.
\end{align*}
%Here, we put $(\lambda^{(i+1)})_1=0$ if $\lambda^{(i+1)}$ is empty (i.e., $\ell(\lambda^{(i+1)})=0$).
\end{Thm}

This result is credited to Mathas in ~\cite[\S9]{AKT}.
%Note that in ~\cite[Proposition 9.7]{AKT}, the set of 2-restricted partitions are mentioned instead of $\STR$.
It is also a Corollary of ~\cite[Theorem 3.8]{KK} and ~\cite[Theorem 10.1]{Lit}.
An element of the connected component $S_k$ is called a Kleshchev multipartition in the context of
the representation theory of Hecke algebras.
For a generalization to $A^{(1)}_{p}$-crystal, where $p\geq 2$, see ~\cite[Corollary 9.6]{AKT}.
For a different characterization, see ~\cite{Jac}.
\begin{Thm}\label{mainresult}
For $k\geq 1$, we have
\begin{align*}
\sum_{\boldsymbol{\lambda}\in S_k}x^{\ell(\boldsymbol{\lambda})}q^{|\boldsymbol{\lambda}|}
=\sum_{i_1,\dots,i_k\geq0}
\frac{q^{\sum_{a=1}^{k}a{1+i_a\choose 2}+\sum_{1\leq a<b\leq k}ai_ai_b}}{(q;q)_{i_1}\cdots(q;q)_{i_k}}x^{\sum_{a=1}^{k}ai_a}.
\end{align*}
Here, for a $k$-tuple of strict partitions $\boldsymbol{\lambda}=(\lambda_1,\dots,\lambda_k)\in\STR^k$,
the size $|\boldsymbol{\lambda}|$ and the length $\ell(\boldsymbol{\lambda})$ are defined as follows.
\begin{align*}
|\boldsymbol{\lambda}|=|\lambda_1|+\cdots+|\lambda_k|,\quad
\ell(\boldsymbol{\lambda}) = \ell(\lambda_1)+\cdots+\ell(\lambda_k).
\end{align*}
\end{Thm}

\section{A proof of Theorem \ref{mainresult}}
As usual (see \cite[Definition 3.1]{An1}), we 
define the $q$-binomial coefficient 
\begin{align*}
{n\brack m}_q=\frac{(q;q)_n}{(q;q)_m(q;q)_{n-m}}
\end{align*} 
for $n\geq m\geq 0$. It is well-known (see ~\cite[Theorem 3.1]{An1}) that we have
\begin{align}
{n\brack m}_q=\sum_{\substack{\lambda\in\PAR \\ \ell(\lambda)\leq m \\ \lambda_1\leq n-m}}q^{|\lambda|}.
\label{LEC}
\end{align}

For $i,j\geq 0$, considering the staircase $\Delta_j=(j,j-1,\dots,1)\in\STR$, we see
\begin{align}
\sum_{\substack{\mu\in\STR \\ \ell(\mu)= j \\ \mu_1\leq i+j}}q^{|\mu|}=
q^{|\Delta_j|}\sum_{\substack{\lambda\in\PAR \\ \ell(\lambda)\leq j \\ \lambda_1\leq i}}q^{|\lambda|}.%\left(=q^{\Delta_j}{i+j\brack j}_q\right).
\label{LEA}
\end{align}
%for $i,j\geq 0$.
%which is equal to $q^{\Delta_j}{i+j\brack j}_q$.
%By $\Delta_i+\Delta_j+ij=\Delta_{i+j}$, we get the result.

\begin{Prop}
For $k\geq 1$ and $j_1,\dots,j_k\geq 0$, there is a size preserving bijection 
\begin{align*}
f_{j_1,\dots,j_k}:V_{j_1,\dots,j_k}\to W_{j_1,\dots,j_k},
\end{align*}
where
\begin{align*}
A_{j_1,\dots,j_k} &= \{\boldsymbol{\lambda}=(\lambda^{(1)},\dots,\lambda^{(k)})\in \STR^k\mid
\ell(\lambda^{(i)})=j_i+\dots+j_k\textrm{ for $1\leq i\leq k$}
\},\\
V_{j_1,\dots,j_k} &= S_k\cap A_{j_1,\dots,j_k},\\
W_{j_1,\dots,j_k} &= \{
\boldsymbol{\lambda}\in A_{j_1,\dots,j_k}\mid
((\lambda^{(i)})_{j_i+1},\dots,(\lambda^{(i)})_{\ell(\lambda^{(i)})})=\Delta_{\ell(\lambda^{(i+1)})}\textrm{ for $1\leq i<k$}
\}.
\end{align*}
\label{LEB}
\end{Prop}

\begin{proof}
We prove the claim by induction on $k$. The case $k=1$ is trivial.

Similarly to ~\eqref{LEA}, for $i,j\geq 0$ we see 
\begin{align}
\sum_{(\lambda,\mu)\in V_{i,j}}q^{|\lambda|+|\mu|} &= \sum_{\substack{\mu\in\STR \\ \ell(\mu)=j \\ \mu_1\leq i+j}} q^{|\mu|} 
\frac{q^{|\Delta_{i+j}|}}{(q;q)_{i+j}},\\
\sum_{(\lambda,\mu)\in W_{i,j}}q^{|\lambda|+|\mu|} &= \frac{q^{|\Delta_{i+j}|}}{(q;q)_i}\frac{q^{|\Delta_j|}}{(q;q)_j},
\label{LED}
\end{align}
which are equal to each other thanks to ~\eqref{LEC} and ~\eqref{LEA}.
This settled the case $k=2$.

For $k\geq 3$, it is easily seen that the composite 
\begin{align*}
\boldsymbol{\lambda}=(\lambda^{(1)},\dots,\lambda^{(k)}) &\mapsto
\boldsymbol{\mu}=(\mu^{(1)},\dots,\mu^{(k)}):=(f_{j_1,j_2+\dots+j_k}(\lambda^{(1)},\lambda^{(2)}),\lambda^{(3)},\dots,\lambda^{(k)}) \\
&\mapsto
(\mu^{(1)},f_{j_2,\dots,j_k}(\mu^{(2)},\dots,\mu^{(k)}))
\end{align*}
is a size preserving bijection from $V_{j_1,\dots,j_k}$ to $W_{j_1,\dots,j_k}$.
\end{proof}

Theorem \ref{mainresult} is proved as follows. Clearly, we have
\begin{align*}
\sum_{\boldsymbol{\lambda}\in S_k}x^{\ell(\boldsymbol{\lambda})}q^{|\boldsymbol{\lambda}|}
=
\sum_{j_1,\dots,j_k\geq 0}x^{j_1+2j_2+\cdots+kj_k}\sum_{\boldsymbol{\lambda}\in V_{j_1,\dots,j_k}}q^{|\boldsymbol{\lambda}|}.
\end{align*}

By Proposition \ref{LEB}, the right hand side is equal to
\begin{align*}
\sum_{j_1,\dots,j_k\geq 0}x^{j_1+2j_2+\cdots+kj_k}\sum_{\boldsymbol{\lambda}\in W_{j_1,\dots,j_k}}q^{|\boldsymbol{\lambda}|}.
\end{align*}
Similarly to ~\eqref{LED}, we see
\begin{align*}
\sum_{\boldsymbol{\lambda}\in W_{j_1,\dots,j_k}}q^{|\boldsymbol{\lambda}|}
=\prod_{a=1}^{k}\frac{q^{|\Delta_{j_a+\cdots+j_k}|}}{(q;q)_{j_a}}.
\end{align*}
Using $|\Delta_{s+t}|=|\Delta_{s}|+|\Delta_{t}|+st$ for $s,t\geq 0$, we have
\begin{align*}
|\Delta_{j_{a}+\cdots+j_k}|
=\sum_{b=a}^{k}|\Delta_{j_b}|+\sum_{a\leq b<b'\leq k}j_bj_{b'}
\end{align*}
and thus we have
\begin{align*}
\sum_{a=1}^{k}|\Delta_{j_{a}+\cdots+j_k}|
=\sum_{a=1}^{k}a|\Delta_{j_a}|+\sum_{1\leq b<b'\leq k}bj_bj_{b'}.
\end{align*}

\section{A proof of the second Rogers-Ramanujan identity}
In the proof, let
\begin{align*}
F(x,q) &= \sum_{s,t,u\geq 0}\frac{q^{{s+1\choose 2}+2{t+1\choose 2}+3{u+1\choose 2}+st+su+2tu}x^{s+2t+3u}}{(q;q)_s(q;q)_t(q;q)_u},\\
G(x,q) &= \sum_{s\geq 0}\frac{q^{s(s+1)}x^{2s}}{(q;q)_s}.
\end{align*}

%\begin{align*}
%\sum_{s,t,u\geq 0}\frac{q^{{s+1\choose 2}+2{t+1\choose 2}+3{u+1\choose 2}+st+su+2tu}x^{s+2t+3u}}{(q;q)_s(q;q)_t(q;q)_u}
%=
%\left(\sum_{s\geq 0}\frac{q^{2{s+1\choose 2}}x^{2s}}{(q;q)_s}\right)
%\left(\sum_{t\geq 0}\frac{q^{t^2}x^{t}}{(q^2;q^2)_t}\right)
%\left(\sum_{u\geq 0}\frac{q^{2{u+1\choose 2}}x^{u}}{(q^2;q^2)_u}\right).
%\end{align*}

\begin{Prop}\label{gdiff}
We have the following $q$-difference equation.
\begin{align*}
G(x,q) = (1+x^2q^2+x^2q^3)G(xq,q)-x^4q^7G(xq^2,q).
\end{align*}
\end{Prop}

\begin{proof}
It is easy to verify that for all $M\in\mathbb{Z}$ we have
\begin{align*}
(1-q^M)g_M -q^M(1+q)g_{M-2} +q^{2M-1}g_{M-4}=0,
\end{align*}
where $g_{2s}=q^{s(s+1)}/(q;q)_s$ for $s\in\mathbb{Z}_{\geq 0}$
and $g_M=0$ for $M\in\mathbb{Z}\setminus 2\mathbb{Z}_{\geq 0}$.
\end{proof}

\begin{Prop}\label{fdiff}
We have the following $q$-difference equation.
\begin{align*}
F(x,q) = (1+xq)(1+x^2q^2+x^2q^3)F(xq,q)-x^4q^7(1+xq)(1+xq^2)F(xq^2,q).
\end{align*}
\end{Prop}

\begin{proof}
Our proof is a typical application of a $q$-version of Wegschaider's 
improvement of Sister Celine's technique (see ~\cite{Rie}).

Let $F(x)=\sum_{n\in\mathbb{Z}}f_n(q)x^n$ and put
\begin{align*}
f(n,t,u) = \frac{q^{{n-2t-3u+1\choose 2}+2{t+1\choose 2}+3{u+1\choose 2}+(n-2t-3u)(t+u)+2tu}}{(q;q)_{n-2t-3u}(q;q)_t(q;q)_u}
\end{align*}
for $n,t,u\in\mathbb{Z}$, where we regard $\frac{1}{(q;q)_{v}}=0$ if $v<0$.
Because $f(n,t,u)$ is $q$-proper hypergeometric (see ~\cite[\S2.1]{Rie}), one can automatically derive
a $q$-holonomic recurrence for $f_n$ thanks to $f_n=\sum_{t,u\in\mathbb{Z}^2}f(n,t,u)$.

Let $(Ng)(n,t,u) = g(n-1,t,u)$,
$(Tg)(n,t,u) = g(n,t-1,u)$, 
$(Ug)(n,t,u) = g(n,t,u-1)$ be the shift operators for $g:\mathbb{Z}^3\to\mathbb{Q}(q)$ and let
\begin{align*}
A &= (1-q^n)-q^nN-q^n(1+q)(N^2+N^3)+q^{2n-1}N^4+q^{2n-2}((1+q)N^5+N^6),\\
B &= (q^n-q^{2t+u})+q^n(-1+q^t+q^u)N+q^{n+u}(1+q^{1+t})N^2+q^{n+2t+u}UN^3,\\
C &= q^n(1-q^t)N+q^{1+n+u}(1-q^t)N^2+q^n(1+q^{1+u})N^3-q^{2n-1}N^4-q^{2n-2}((1+q)N^5+N^6).
%C &= q^n(1-q^t)N^{-1}+q^{1+n+u}(1-q^t)N^{-2}+q^n(1+q^{1+u})N^{-3}-q^{2n-1}N^{-4}-q^{2n-2}(1+q)N^{-5}-q^{2n-2}N^{-6},
\end{align*}
One can check that 
\begin{align*}
(A+(1-T)B+(1-U)C)f(n,t,u)=0.
\end{align*}
By this certificate recurrence operator (see ~\cite[\S3]{Rie} and ~\cite[\S7.1]{Tsu}), we get
\begin{align*}
(1-q^n)f_n-q^nf_{n-1}-q^n(1+q)(f_{n-2}+f_{n-3})+q^{2n-1}f_{n-4}+q^{2n-2}((1+q)f_{n-5}+f_{n-6})=0
\end{align*}
for $n\in\mathbb{Z}$. This is equivalent to the $q$-difference equation in the Proposition.
\end{proof}

\begin{Cor}\label{fincor}
We have $F(x,q)= (-xq;q)_{\infty}G(x,q)$.
\end{Cor}

\begin{proof}
By Proposition \ref{gdiff} and Proposition \ref{fdiff}, $F(x,q)$ and $(-xq;q)_{\infty}G(x,q)$
satisfy the same $q$-difference equation presented in Proposition \ref{fdiff}.
Then, the equality follows from the fact that both
the coefficients of $x^0$ (resp. $x^n$ for $n<0$) in $F(x,q)$ and $(-xq;q)_{\infty}G(x,q)$ are equal to 1 (resp. 0).
\end{proof}

\begin{Rem}
After submission to arXiv of the first version of this paper, 
we learned from Ole Warnaar that Corollary \ref{fincor} is easily deduced
by a trick to use $f_n=\sum_{t,u\in\mathbb{Z}} f(n,t-u,u)$ instead of $f_n=\sum_{t,u\in\mathbb{Z}} f(n,t,u)$ noticing
\begin{align*}
f(n,t-u,u)=\frac{q^{{n+1\choose 2}+t(t-n)}}{(q;q)_{n-2t}(q;q)_{t}}
\frac{(q^{-t};q)_u(q^{-(n-2t)};q)_u}{(q;q)_u}.
\end{align*}
Thanks to the $q$-Chu-Vandermonde identity 
${}_2\phi_1(a,q^{-m};0;q,q)=a^m$ for a nonnegative integer $m$ 
(see ~\cite[(2.41)]{An2}), 
we have
\begin{align*}
f_n
=\sum_{t=0}^{\lfloor n/2\rfloor}\frac{q^{{n+1\choose 2}+t(t-n)-t(n-2t)}}{(q;q)_{n-2t}(q;q)_{t}}
=\sum_{t=0}^{\lfloor n/2\rfloor}\frac{q^{{n-2t+1\choose 2}+t(t+1)}}{(q;q)_{n-2t}(q;q)_{t}}.
\end{align*}
This is equivalent to Corollary \ref{fincor} by Euler's identity $(-xq;q)=\sum_{m\geq 0}\frac{q^{{m+1\choose 2}}x^m}{(q;q)_m}$.
\end{Rem}

The second Rogers-Ramanujan identity \eqref{RRiden} is proved as follows.
By Theorem \ref{mainresult}, Lepowsky-Milne's observation \eqref{chsimi} is translated to
\begin{align*}
F(1,q)=\frac{1}{(q;q^2)_{\infty}}\frac{1}{(q^2,q^3;q^5)_{\infty}}.
\end{align*}

By Corollary \ref{fincor} and Euler's identity $(q;q^2)_{\infty}(-q;q)_{\infty}=1$, we have
\begin{align*}
G(1,q)=\frac{1}{(q^2,q^3;q^5)_{\infty}}.
\end{align*}

\noindent{\bf Acknowledgments.} 
The author would like to express his sincere thanks to 
Ole Warnaar for sharing a simple proof with the author.
%and give the author permission to include it in this paper.
The author was supported by JSPS Kakenhi Grant 20K03506, the Inamori Foundation, JST CREST Grant Number JPMJCR2113, Japan and Leading Initiative for Excellent Young Researchers, MEXT, Japan.

%\bibliographystyle{abbrv}
%\bibliography{clean}

\begin{thebibliography}{99}

\bibitem{AKT}
S.~Ariki, V.~Kreiman and S.~Tsuchioka, 
\textit{On the tensor product of two basic representations of $U_v(\mathfrak{sl}_e)$},
Adv.Math. \textbf{218} (2008), 28--86. 

\bibitem{An1}
G.E.~Andrews, \textit{The theory of partitions}, 
Encyclopedia of Mathematics and its Applications, vol.2, Addison-Wesley, 1976.

\bibitem{An2}
G.E.~Andrews, \textit{$q$-series: their development and application in analysis, number theory, combinatorics, physics, and computer algebra},
CBMS Regional Conference Series in Mathematics, 66. Published for the Conference Board of the Mathematical Sciences, Washington, DC; by the American Mathematical Society, Providence, RI, 1986.

\bibitem{Cor} 
S.~Corteel, 
\textit{Rogers-Ramanujan identities and the Robinson-Schensted-Knuth correspondence},
Proc.Amer.Math.Soc. \textbf{145} (2017), 2011--2022.

\bibitem{DL} 
J.~Dousse and J.~Lovejoy, 
\textit{On a Rogers-Ramanujan type identity from crystal base theory},
Proc.Amer.Math.Soc. \textbf{146} (2018), 55--67. 

%\bibitem{Gor}
%B.~Gordon, \textit{A combinatorial generalization of the Rogers-Ramanujan identities}, 
%Amer.J.Math. \textbf{83} (1961) 393--399.

%\bibitem{Har}
%G.H.~Hardy, \textit{Ramanujan}, Cambridge University Press, 1940.

\bibitem{Jac}
N.~Jacon, 
\textit{Kleshchev multipartitions and extended Young diagrams}, 
Adv.Math. \textbf{339} (2018), 367--403. 

\bibitem{KK}
M.~Kim and S.~Kim, 
\textit{On the connection between Young walls and Littelmann paths},
J.Algebra \textbf{323} (2010), 2326--2336. 

\bibitem{Lit}
P.~Littelmann, 
\textit{A plactic algebra for semisimple Lie algebras},
Adv.Math. \textbf{124} (1996), 312--331. 

\bibitem{LM}
J.~Lepowsky and S.~Milne, 
\textit{Lie algebraic approaches to classical partition identities},
Adv.Math. \textbf{29} (1978), 15--59. 

\bibitem{LW3}
J.~Lepowsky and R.~Wilson, 
\textit{The structure of standard modules. I. Universal algebras and the Rogers-Ramanujan identities}, 
Invent.Math. \textbf{77} (1984), 199--290.

\bibitem{MM}
K.~Misra and T.~Miwa,
\textit{Crystal base for the basic representation of $U_q(\mathfrak{sl}(n))$}, 
Comm.Math.Phys. \textbf{134} (1990), 79--88.

\bibitem{Rie}
A.~Riese, 
\textit{qMultiSum -- a package for proving q-hypergeometric multiple summation identities}, 
J.Symbolic Comput. \textbf{35} (2003), 349--376.

%\bibitem{Sil}
%A.V.~Sills, 
%\textit{An invitation to the Rogers-Ramanujan identities. With a foreword by George E. Andrews}, 
%CRC Press, (2018).

\bibitem{Tsu}
S.~Tsuchioka,
\textit{An example of $A_2$ Rogers-Ramanujan bipartition identities of level 3}, 
arXiv:2205.04811


\end{thebibliography}

\end{document}